\documentclass[10pt, article]{amsart}
\usepackage{tikz}
\usetikzlibrary{calc}
\usepackage{ae} 
\usepackage[T1]{fontenc}
\usepackage[cp1250]{inputenc}
\usepackage{amsmath}
\usepackage{amssymb, amsfonts,amscd,verbatim}

\usepackage[normalem]{ulem}
\usepackage{hyperref}
\usepackage{indentfirst}
\usepackage{latexsym}
\input xy
\xyoption{all}

\usepackage{xcolor}

\usepackage{amsmath}    

\theoremstyle{plain}
\newtheorem{Pocz}{Poczatek}[section]
\newtheorem{Proposition}[Pocz]{Proposition}

\newtheorem{Theorem}[Pocz]{Theorem}
\newtheorem{Corollary}[Pocz]{Corollary}

\newtheorem{Lemma}[Pocz]{Lemma}
\newtheorem{Observation}[Pocz]{Observation}

\newtheorem{Example}[Pocz]{Example}

\theoremstyle{definition}
\newtheorem{Definition}[Pocz]{Definition}

\theoremstyle{remark}
\newtheorem{Remark}[Pocz]{Remark}

\numberwithin{equation}{section}
\title[Boundaries for geodesic spaces]
{Boundaries for geodesic spaces}

\author{Jerzy Dydak}
\address{University of Tennessee, Knoxville, TN 37996, USA}
\email{jdydak@utk.edu}

\author{Hussain Rashed}
\address{University of Tennessee, Knoxville, TN 37996, USA}
\email{hrashed1@vols.utk.edu}

\date{ \today}
\keywords{contracting boundary, Morse boundary, boundary at infinity, contracting geodesic, Gromov boundary, quasi-geodesics, quasi-isometries, sublinearly Morse boundary}

\subjclass[2000]{Primary 20F65; Secondary 20F67}

\begin{document}
\maketitle
\begin{center}
\today
\end{center}

\tableofcontents

\begin{abstract}
For every proper geodesic space $X$ we introduce its quasi-geometric boundary $\partial_{QG}X$
with the following properties:\\
1. Every geodesic ray $g$ in $X$ converges to a point of the boundary $\partial_{QG}X$
and for every point $p$ in $\partial_{QG}X$ there is a geodesic ray in $X$ converging to $p$,\\
2. The boundary $\partial_{QG}X$ is compact metric,\\
3. The boundary $\partial_{QG}X$ is an invariant under quasi-isometric equivalences,\\
4. A quasi-isometric embedding induces a continuous map of quasi-geodesic boundaries,\\
5. If $X$ is Gromov hyperbolic, then $\partial_{QG}X$ is the Gromov boundary of $X$.\\
6. If $X$ is a Croke-Kleiner space, then $\partial_{QG}X$ is a point.
\end{abstract}

\section{Introduction}

After Croke-Kleiner \cite{CK} discovered that the visual boundary of CAT(0)-spaces is not an invariant of quasi-isometries, several authors embarked on extending the idea of Gromov boundary to other proper geodesic spaces. The first step was due to Charney-Sultan \cite{CS} who extracted the essential feature of geodesics in hyperbolic spaces, namely being \textbf{contracting}, and defined the so-called \textbf{contracting boundary} of proper CAT(0)-spaces. Their construction was extended by Cashen-Mackay \cite{CM} to arbitrary proper geodesic spaces. The generalization which is best known and most used is the Morse boundary, which is due to Matthew Cordes \cite{MC}.
Finally, Tiozzo-Qing-Rafi
\cite{TQR1}  and \cite{TQR2} introduced the so-called \textbf{sublinear Morse boundary}. Here are its main properties:

 \begin{Theorem}\cite{TQR2} 
Let $X$ be a proper, geodesic metric space, and let $\kappa$ be a sublinear function. Then we construct a topological space $\partial_{\kappa} X$ with the following properties:\\
(1) (Metrizability) The spaces $\partial_{\kappa} X$ and $X\cup \partial_{\kappa} X$ are metrizable, and $X\cup\partial_{\kappa} X$ is a bordification of X;\\
(2) (QI-invariance) Every $(k,K)$-quasi-isometry $\Phi: X\to Y$ between proper geodesic metric spaces induces a homeomorphism  $\Phi^\ast: \partial_{\kappa} \to \partial_{\kappa} Y$;\\
(3) (Compatibility) For sublinear functions $\kappa$ and $\kappa'$, where  $\kappa\leq  c\cdot \kappa'$ for some $c > 0$, we have $\partial_{\kappa} X\subset \partial_{\kappa'} X$ where the topology of $\partial_{\kappa} X$ is the subspace topology. Further, letting 
$ \partial X :=  \bigcup \partial_{\kappa} X$, we obtain a quasi-isometrically invariant topological space that contains all $\partial_{\kappa} X$ as topological subspaces. We call $\partial X$ the sublinearly Morse boundary of $X$.
\end{Theorem}

The essential feature of the above constructions is that they are of covariant nature (see \cite{Dyd}).
Namely, the authors probe a space $X$ by mapping objects into it, which initially makes sense as geodesics are covariant objects. In this paper we apply the contravariant approach and probe $X$ by mapping it to spaces. That leads to better features of the \textbf{quasi-geodesic boundary} defined that way. For instance, the contracting boundary of the plane is empty while its quasi-geodesic boundary consists of one point. Also, one can add the quasi-geodesic boundary to $X$ and get a metrizable compactification of $X$.

We do not know the precise relation between our quasi-geodesic boundary and the sublinearly Morse boundary of $X$. They are clearly different if the sublinearly Morse boundary of $X$ is not compact or is not metrizable.
It seems sublinearly Morse boundary is useful to study Poisson boundary (see \cite{TQR1}  and \cite{TQR2}). 

We are able to define a sublinear version of our boundary as well.
We do not know if boundaries we are constructing have any connection to the Martin boundary.

We are grateful to Ruth Charney and Yulan Qing for helpful comments that improved the exposition of proofs in the paper. We are grateful to organizers of the online workshop "A Week At Infinity" (March 28-April 1, 2022) which introduced us to topics related to Morse boundaries. We are grateful to Kim Ruane for suggesting future lines of research.

\section{Integral quasi-geodesics}

Traditionally, researchers dealt with quasi-geodesics defined on intervals in reals. We find it easier to use quasi-geodesics defined on intervals in positive integers. Notice it is easy two switch back and forth between traditional quasi-geodesics and integral quasi-geodesics in geodesic spaces. However, the advantage of integral quasi-geodesics is that, in the case of finitely generated groups, one does not need to extend them to their Cayley graphs.

\begin{Definition}
An \textbf{integral $(q,Q)$-quasi-geodesic} on a metric space $(X,d)$ is a function
$g:(a,b)\cap Z_+\to X$ such that for any $x,y\in (a,b)\cap Z_+$
$$\frac{|x-y|}{q}-Q\leq d(g(x),g(y))\leq q\cdot |x-y|+Q.$$

If every two points of $X$ can be joined by an integral $(q,Q)$-quasi-geodesic,
we say that $X$ is \textbf{integrally $(q,Q)$-quasi-geodesic}.
$X$ is \textbf{integrally quasi-geodesic} if it is integrally $(q,Q)$-quasi-geodesic
for some $q\ge 1$ and $Q\ge 0$.
\end{Definition}

\begin{Observation}
Notice any finite union of rays from the origin of the plane is integrally quasi-geodesic but is geodesic only if the number of rays is $1$. Thus, the concept of $X$ being integrally quasi-geodesic covers more spaces than being geodesic. Also, it is a quasi-isometric invariant.
\end{Observation}

\begin{Definition}
The \textbf{fan $F(x_0,q,Q)$ of integral $(q,Q)$-quasi-geodesics} is the family of all integral $(q,Q)$-quasi-geodesics $g$ in $X$ originating at $x_0$, i.e. $g(1)=x_0$.

Suppose $m\ge 1$ and $g\in F(x_0,q,Q)$ is an integral quasi-geodesic ray. The \textbf{$m$-thread $T(x_0,q,Q,g,m)$} of the fan $F(x_0,q,Q)$ based at $g\in F(x_0,q,Q)$ is the subfamily of the fan containing  all elements $h$ of $F(x_0,q,Q)$ such that $h(i)=g(i)$ for each $i\leq m$.

Suppose $m\ge 1$ and $g\in F(x_0,q,Q)$ is an integral quasi-geodesic ray. The \textbf{$m$-cone $C_X(x_0,q,Q,g,m)$} is the set of all points $x\in X$ such that there is an integral $(q,Q)$-quasi-geodesic $h$ starting at $x_0$, agreeing with $g$ for $i\leq m$
and satisfying $h(i)=x$ for some $i \ge m$.
\end{Definition}

\begin{Example}\label{ConeExample}
Consider the $(1,0)$-quasi-geodesic $g$ on the plane
given by $g(i)=(i,0)$ for $i\ge 1$. $C_X(x_0,g,2,2,m)$
contains all points $(x,y)$ on the plane such that $y\ge m$.
\end{Example}

\begin{Lemma}\label{SequenceVsRayLemma}
Suppose $k\ge 1$ and $X$ is a metric space such that for some $M > 0$ every bounded subset of $X$ can be covered by finitely many sets of diameter less than $M$. Suppose $d(x_0,x_n)\to\infty$
and $d(x_n,y_n) < M$ for each $n\ge 1$.
If $g_n$ is a sequence of integral $(k,k)$-quasi-geodesics joining $x_0$ to $x_n$, then there is an integral
$(k,k+2M)$-quasi-geodesic ray $g$ and a sequence
of integral $(k,k+2M)$-quasi-geodesics $h_n$
such that for some subsequence $k(n)$ of $Z_+$
the last two values of $h_n$ are $x_{k(n)}$ and $y_{k(n)}$
and each $m$-thread $T(x_0,k,k+2M,g,m)$
contains all but finitely many $h_n$'s.
\end{Lemma}
\begin{proof}
By induction choose a decreasing sequence $S_n$ of infinite subsets of $Z_+$ with the property that all points $g_j(n)$, $j\in S_n$,
belong to the same set $B_n$ of diameter less than $M$. Indeed, put $B_1:=\{x_0\}$, and $S_1:=\{j\in Z_+:g_j(1)=x_0\}$. Assume that $n\in Z_+$ such that $S_n:=\{j\in Z_+:g_j(n)\in B_n\}$ is infinite, for some bounded subset $B_n$ of diameter less than $M$. Notice that  $g_j(n+1)\in B(B_n,2\cdot k)$, for every $j\in S_n$. Since $B(B_n,2\cdot k)$ can be covered by finitely many bounded sets of diameter less than $M$, we can pick a bounded set $B_{n+1}$ of diameter less than $M$ and $S_{n+1}:=\{j\in Z_+:g_j(n+1)\in B_{n+1}\}$ is infinite and contained in $S_n$.

Pick $z_n\in B_n$ for each $n\ge 1$ and notice
$g(n)=z_n$ defines an integral
$(k,k+2M)$-quasi-geodesic ray. Indeed, if $j < m < n$,
then $d(g(i),g_p(i)) < M$ for all $i\in [j,m]$ if $p\in S_n$,
so $d(g_p(j),g_p(m) - 2M < d(g(j),g(m)) < d(g_p(j),g_p(m) + 2M$. Recall $g_p$ is an integral $(k,k)$-quasi-geodesic,
so $(1/k)(m-j)-k-2M < d(g(j),g(m))< k(m-j)+k+2M$
and $g$ is an integral $(k,k+2M)$-quasi-geodesic ray.

Pick an increasing sequence $k(n)\in S_n$, $k(n) > n$, for each $n\ge 1$, such that
$g_{k(n)}$ is defined on integers in $[1,m(n)]$ with $m(n+1) > m(n)$ and $g_{k(n)}(m(n))=x_{k(n)}$ for each $n\ge 1$.
Define
$h_n$ by $h_n(i)=z_i$ for $i\leq n$,
$h_n(i)=g_{k(n)}(i)$ for $m(n)\ge i > n$, 
then extend it by declaring $h_n(m(n)+1)=y_{k(n)}$.
\end{proof}

\begin{Observation}\label{ExistenceOfGeodesicRays}
If $X$ is geodesic and proper, then we may pick $g$
to be induced by a geodesic ray in $X$.
\end{Observation}

\begin{Proposition}\label{ProductOfQuasiGeodesicSpaces}
If two proper metric spaces $(X,d_X)$ and $(Y,d_Y)$ are integrally quasi-geodesic, then their cartesian product $X\times Y$ is integrally quasi-geodesic in either the $l_1$-metric or in the $l_2$-metric.
\end{Proposition}
\begin{proof}
By the $l_1$-metric on $X\times Y$ we mean $d_1$ defined
as $d_1((x_1,y_1),(x_2,y_2)):=d_X(x_1,x_2)+d_Y(y_1,y_2)$.
By the $l_2$-metric on $X\times Y$ we mean $d_2$ defined
as $d_2((x_1,y_1),(x_2,y_2)):=\sqrt{d^2_X(x_1,x_2)+d^2_Y(y_1,y_2)}$. Since $(X\times Y,d_1)$ is quasi-isometric to
$(X\times Y,d_2)$, it suffices to prove $(X\times Y,d_1)$ is integrally quasi-geodesic. 

Suppose $(X,d_X)$ and $(Y,d_Y)$ are integrally $(k,k)$-quasi-geodesic for some $k\ge 1$. Given $(x_1,y_1), (x_2,y_2)\in X\times Y$, choose
a $(k,k)$-quasi-geodesic $f:[0,a]\to X$ from $x_1$ to $x_2$
and a $(k,k)$-quasi-geodesic $g:[0,b]\to Y$ from $y_1$ to $y_2$.
Define $h:[0,a+b]\to X\times Y$ as follows:\\
$h(t):=(f(t),y_1)$ for $t\leq a$ and $h(t):=(x_2,g(t-a))$ for $a\leq t\leq a+b$. Given $t\leq a$ and $s\ge a$, we need to estimate
$|h(s)-h(t)|$.

$|h(s)-h(t)|\leq |h(s)-h(a)|+|h(a)-h(t)|=d_Y(g(s-a),g(0))+d_X(f(t),f(a))\leq k+k\cdot (s-a)+k+k\cdot(a-t)=2k+k\cdot (s-t)$.

If $|h(s)-h(t)| < (s-t)/(2\cdot k)-2k$ for some $t\leq a$ and $s\ge a$, then
$(a-t)/k - k \leq d_X(f(t),x_2) < (s-t)/(2\cdot k)-2k$ and
$(s-a)/k-k \leq (d_Y(y_1,g(s-a)) < (s-t)/(2\cdot k)-2k$.
Adding both sides of inequalities, we get
$(s-t)/k-2k < (s-t)/k-4k$, a contradiction.

Thus, $(X\times Y,d_1)$ is integrally $(2k,2k)$-quasi-geodesic.
\end{proof}

\section{Geometrically slowly oscillating functions}

Recall (see \cite{Roe lectures}) that a function $f:X\to \mathbb{C}$
is \textbf{slowly-oscillating} if, whenever $x_n\to\infty$ and there is $M > 0$
such that $d(x_n,y_n) < M$ for all $n\ge 1$, then $|f(x_n)-f(y_n)|\to 0$.

\begin{Definition}\label{GeoSlowOscDef}
Suppose $X$ is a metric space.
A bounded continuous slowly-oscillating function $f:X\to \mathbb{C}$ is \textbf{geometrically slowly-oscillating} if for every $\epsilon > 0$ and every integral $(k,k)$-quasi-geodesic ray
$g$ in $X$ originating at $x_0\in X$ there is $m\ge 1$ such that the diameter of 
$\{f(h(i))\mid i \ge m, h\in T(x_0,k,k,g,m)\}$ is smaller than $\epsilon$. Equivalently, the diameter of 
$f(C_X(x_0,k,k,g,m))$ is smaller than $\epsilon$. 
\end{Definition}

\begin{Observation}\label{GeoSloOscAreSlowOscInGeodesicSpaces}
If $X$ is a metric space, then every vanishing at infinity function $f:X\to \mathbb{C}$ is geometrically slowly-oscillating function; where a function  $f:X\to \mathbb{C}$ is said to vanish at infinity if for $\epsilon$, there exists a bounded subset $K\subset X$ such that $|f(x)|<\epsilon$ for all $x\in X\setminus K$. If, moreover, for some $M > 0$, every bounded subset of $X$ can be covered by finitely many sets of diameter less than $M$ and $X$ is integrally quasi-geodesic,
then every geometrically slowly-oscillating function $f$ on $X$
is slowly-oscillating by \ref{SequenceVsRayLemma}.
\end{Observation}

\begin{Example}\label{CoarselyClopenGeoSlo}
Suppose $X$ is a metric space and $U$ is a coarsely clopen subset of $X$. Any bounded continuous function $f:X\to \mathbb{C}$
such that there is a bounded subset $B$ so that both $f|(U\setminus B)$ and $f|(U^c\setminus B)$
are constant is geometrically slowly oscillating.
\end{Example}
\begin{proof}
$U$ being coarsely clopen means that for each $r > 0$
there is a bounded subset $B_r$ of $X$ such that $d(x,y) > r$
if $x\in U\setminus B_r$ and $y\in U^c\setminus B_r$.

Suppose $\epsilon > 0$ and $g$ is an integral $(k,k)$-quasi-geodesic ray
$g$ in $X$ originating at $x_0\in X$.
Choose $t > 0$ such that $B(x_0,t)$ contains $B$ and $B_{2k}$.
Let $m > (k+t)\cdot k+1$.
Given $h\in T(x_0,k,k,g,m)$, $h(m)\in U\setminus B_t$
we claim $h(i)\in U\setminus B(x_0,t)$ for all $i > m$.
Consider the smallest $i > m$ so that $h(i)\in B(x_0,t)$.
Now, $(1/k)\cdot (i-1)- k < d(h(i),h(1)) < t$
resulting in $i < k\cdot (k+t)+1=m$, a contradiction.
Notice, for $k\in (m,i)$, $h(k)\in U$ as otherwise we would witness
a jump from $U$ to $U^c$ outside of $B_{2k}$, a contradiction. 
That shows
$\{f(h(i))\mid i \ge m, h\in T(x_0,k,k,g,m)\}$ is a single point.
\end{proof}

By Urysohn's Lemma, the family $C_{QG}(X)$ of all continuous bounded geometrically slowly oscillating functions on $X$ separates points from closed sets.

\begin{Corollary}
Given a geometrically slowly-oscillating function $f:X\to \mathbb{C}$
and given a quasi-geodesic ray $g$ on $X$, $f\circ g$ is a Cauchy sequence and therefore it converges to a point in $R$ that we will denote by $f(g)$.
\end{Corollary}

\begin{Definition}\label{BoundedlyApproachesDef}
Given $k\ge 1$, a sequence
$g_n$ of integral $(k,k)$-quasi-geodesics \textbf{boundedly approaches}
$g$ if the following conditions are satisfied:\\
1. The lengths of $g_n$ diverge to infinity,\\
2. There is $C > 0$ such that for each $m\ge 1$ there is 
$N\ge 1$ so that $d(g_n(i),g(i)) < C$ for all $n\ge N$ and all $i\leq m$.

By the \textbf{length of a quasi-geodesic} we mean the size of its domain.
\end{Definition}

\begin{Observation}\label{ObservationOnBdApp}
Notice $g$ is an integral $(k,k+2C)$-quasi-geodesic ray in 
\ref{BoundedlyApproachesDef}. Indeed, if $a < b$,
we can choose $n > b$ such that $d(g_n(i),g(i)) < C$ for all $i\leq b$. Therefore, $d(g(a),g(b)) < d(g_n(a),g_n(b))+2C\leq
k\cdot (b-a)+k+2C$ and $d(g(a),g(b)) > d(g_n(a),g_n(b))-2C\ge
\frac{1}{k}\cdot (b-a)+k-2C$.
\end{Observation}

\begin{Proposition}\label{BoundedlyApproachesProp}
If $f:X\to \mathbb{C}$ is geometrically slowly-oscillating, then $f(g_n)\to f(g)$ for every sequence
$g_n$ of integral $(k,k)$-quasi-geodesic rays boundedly approaching
$g$. If some of $g_n$'s are not quasi-geodesic rays, then $f(g_n)$
means $f(y_n)$, where $y_n$ is the endpoint of $g_n$.
\end{Proposition}
\begin{proof}
Choose $C > 0$ such that for each $m\ge 1$ there is $N\ge 1$ so that $d(g_n(i),g(i)) < C$ for all $n\ge N$ and all $i\leq m$.
Suppose $\epsilon > 0$. Choose $m\ge 1$ such that
the diameter of $\{f(h(i))\mid i \ge M, h\in T(x_0,k+2C,k+2C,g,m)\}$ than $\epsilon$. Once $d(g_n(i),g(i)) < C$ for all $i\leq m$
define $h_n(i)$ as $g(i)$ for $i\leq m$ and $g_n(i)$ for $i > m$.
Notice each $h_n$ is an  integral $(k+2C,k+2C)$-quasi-geodesics
(see \ref{ObservationOnBdApp})
and $f(h_n)=f(g_n)$ for each $n\ge 1$. Notice
all $f(h_n)=f(g_n)$, $n > m$, are within $\epsilon$ from $f(g)$
as $h_n\in T(x_0,k+2C,k+2C,g,m)$ for $n > m$.
\end{proof}

\begin{Proposition}\label{UniformLimitOfGeoSloOsc}
If $f:X\to \mathbb{C}$ is the uniform limit of geometrically slowly-oscillating functions $f_n$, then $f$ is geometrically slowly-oscillating.
\end{Proposition}
\begin{proof}
Notice $f$ is bounded as each $f_n$ is bounded.
Given $\epsilon > 0$ choose $n$ such that $|f_n-f| < \epsilon/2$.
Given an integral $(k,k)$-quasi-geodesic ray
$g$ in $X$ originating at $x_0\in X$ there is 
$m\ge 1$ such that the diameter of
$\{f_n(h(i))\mid i \ge m, h\in f_n(T(x_0,k,k,g,m))\}$
 is less than $\epsilon/2$.
Therefore, the diameter of
$\{f(h(i))\mid i \ge m, h\in f(T(x_0,k,k,g,m))\}$
 is less than $\epsilon$.
\end{proof}

\section{Quasi-geodesic boundaries}
In this section we introduce the quasi-geodesic compactification of metric spaces in a way similar to the Higson compactification. That approach should be of use to researchers in geometric group theory. 
 Let $B_{QG}(X)$ denote the algebra of all geometrically slowly-oscillating functions $f:X\to \mathbb{C}$, $C_{0}(X)$ be the ideal of all geometrically slowly-oscillating functions that vanish at infinity i.e. $f\in B_{0}(X)$ if and only if for every $\epsilon > 0$ there exist a bounded subset $K$ of $X$ such that $|f(x)| < \epsilon$ for all $x\in X\setminus K$.

\begin{Definition}
The \textbf{quasi-geodesic compactification} $X_{QG}$ of a metric space $X$ is the maximal ideal space of the $C^{*}$-algebra $C_{QG}(X)$.
If $X$ is proper, then its \textbf{quasi-geodesic corona} (or its \textbf{quasi-geodesic boundary}) $\partial_{QG}X:=X_{QG}\setminus X$ is a compact subset of $X_{QG}$. Moreover, $\partial_{QG}X$ is the maximal ideal space of the $C^{*}$-algebra ${C_{QG}(X)}/{C_{0}(X)}$.
\end{Definition}

Equivalently, the \textbf{quasi-geodesic compactification} $X_{QG}$ of a metric space $X$ can be seen as the one that has the property that all geometrically slowly-oscillating functions on $X$ extend over it
and every restriction of a continuous function $f:X_{QG}\to R$
to $X$ is geometrically slowly-oscillating. If $X$ is not proper, we can still talk about its quasi-geodesic boundary as the set of points in $X_{QG}$ accessible via quasi-geodesic rays.

\begin{Lemma}\label{DiscreteSubsetLemma}
If $Y$ is a closed subset of a metric space $X$ such that $X=B(Y,r)$ for some $r > 0$, then the closure of $Y$ in the quasi-geometric compactification of $X$ is the quasi-geometric compactification of $Y$ and it contains the quasi-geometric boundary of $X$.
\end{Lemma}
\begin{proof}
Given a geometrically slowly-oscillating function $f:Y\to [0,1]$, we can extend it to a Higson function $g:X\to [0,1]$ by \cite{DW}.
That means $g$ is slowly-oscillating and continuous.
We need to show $g$ is geometrically slowly-oscillating. 

Given $k\ge 1$ let us define a shift operator $S_k$ sending all integral $(k,k)$-quasi-geodesics $a$ in $X$ based at $x_0\in Y$
to integral $(k+2r,k+2r)$-quasi-geodesics $S_k(a)$ in $Y$ based at $x_0$.  First, we choose a retraction $p:X\to Y$ (i.e. a function fixing points of $Y$) such that $d(x,p(x)) < r$ for all $x\in X$.
Then, given $a$, 
for each $i\ge 1$ define
$S_k(a)(i)$ to be equal to $p(a(i))$.

Now,
suppose $\epsilon > 0$ and $a$ is an integral $(k,k)$-quasi-geodesic ray in $X$ originating at $x_0\in Y$. Since $g$ is slowly-oscillating,
there is $t > 0$ such that $|g(x)-g(y)| < \epsilon/8$ if
$x,y\in X\setminus B(x_0,t)$ and $d(x,y) < r$. 
There is $s > 0$ such that the $m$-cone $C_X(x_0,k+2r,k+2r,S_k(a),m)$ is contained in $X\setminus B(x_0,t+r)$
for all $m\ge s$. Pick $m > s$ with the property that the diameter of 
$f(C_Y(x_0,k+2r,k+2r,S_k(a),m))$ is smaller than $\epsilon/8$. 
Given $x\in C_X(x_0,k,k,a,m)$ there is an integral $(k,k)$-quasi-geodesic $b$ agreeing with $a$ on indices up to $m$ such that
$x=b(j)$ for some $j > m$. Put $y:=S_k(b)(j)$ and $z:=S_k(b)(m)$.
Now, $|g(y)-g(z)| < \epsilon/8$ as both $y$ and $z$ belong to
$C_Y(x_0,k+2r,k+2r,S_k(a),m)$ and $g(y)=f(y)$, $g(z)=f(z)$. Also, $d(x,y) < r$, hence $|g(x)-g(y)| < \epsilon/8$.
Finally, $|g(x)-g(z)| <\epsilon/2$,
so the diameter of 
$g(C_X(x_0,k+2r,k+2r,a,m))$ is smaller than $\epsilon$. 

If the closure of $Y$ in $X_{QG}$ misses a point $w$ in $\partial_{QG}X$, then there is a continuous function
$c:X_{QG}\to [0,1]$ sending $cl(Y)$ to $0$ and sending $w$ to $1$. Hence, there is a sequence $x_n$ of points in $X$ diverging to infinity so that $c(x_n) > .5$ for each $n\ge 1$.
As $d(p(x_n),x_n) < r$ and $c|X$ is slowly-oscillating
$c(x_n)=|c(x_n)-c(p(x_n))|\to 0$, a contradiction. 
\end{proof}

\begin{Corollary}
Suppose $k\ge 1$ and $X$ is a metric space such that for some $M > 0$ every bounded subset of $X$ can be covered by finitely many sets of diameter less than $M$.
If $X$ is integrally $(k,k)$-quasi-geodesic, then for every $x_0$
and every integral quasi-geodesic ray $g$ there is an integral
$(k+2M,k+2M)$-quasi-geodesic ray $h$ based at $x_0$ that converges to the same point on the quasi-geodesic boundary of $X$ as $g$ does.
\end{Corollary}
\begin{proof}
Use \ref{SequenceVsRayLemma}.
\end{proof}

\begin{Proposition}\label{BdIsMetrizable}
Suppose $k\ge 1$ and $X$ is a discrete and proper metric space.
If $X$ is integrally $(k,k)$-quasi-geodesic, then its quasi-geodesic boundary is compact metrizable and every point of it is the limit of some
integral $(k,k)$-quasi-geodesic ray starting at a base point.
\end{Proposition}
\begin{proof}
The quasi-geodesic compactification of $X$ was defined in a contravariant way, i.e. by mapping $X$ to $\mathbb{C}$.
The properties of the quasi-geodesic boundary of $X$ will be detected in a covariant way, i.e. by mapping a compact metric space
onto it.

Fix $x_0\in X$ and consider $B_n$, the subset of all points in $X$ reachable from $x_0$ by an integral $(k,k)$-quasi-geodesic defined on
$[1,n]\cap Z_+$. Consider the set $S_n$ of all integral $(k,k)$-quasi-geodesic defined on
$[1,n]\cap Z_+$ and starting at $x_0$. That set is finite, so we put the discrete topology on it.

There are maps $r^{n+1}_n:S_{n+1}\to S_n$ defined
by $r^{n+1}_n(g):=g|[1,n]\cap Z_+$. Its inverse limit
maps surjectively onto all integral $(k,k)$-quasi-geodesic rays in $X$ starting at $x_0$, hence it
maps continuously onto the quasi-geodesic boundary of $X$.
Indeed, suppose $p\in U$, $p\in\partial_{QG} X$ and $U$ is open in
$X_{QG}$. There is a continuous function
$f:X_{QG}\to [0,1]$ such that $f(p)=1$ and $f(X_{QG}\setminus U)\subset {0}$. Suppose $g$ is an integral $(k,k)$-quasi-geodesic
converging to $p$. Since $f|X$ is geometrically slowly-oscillating,
there is $m\ge 1$ such that such that the diameter of 
$\{f(h(i))\mid i \ge m, h\in T(x_0,k,k,g,m)\}$ is smaller than $1/2$. 
Notice $T(x_0,k,k,g,m)$ represents a neighborhood $V$ of $g$ in the inverse limit $M$
and its image in $\partial_{QG} X$ must be contained in $U$ as $f(h)\ge 1/2$
for any $h\in V$.

Suppose that image is not all of the geometric quasi-geodesic boundary of $X$, so there is a non-zero (on $\partial_{QG} X$) geometrically slowly-oscillating
$f:X\to [0,1]$ such that $f(g)=0$ for each integral $(k,k)$-quasi-geodesic ray $g$ in $X$ starting at $x_0$ and $f(p)=1$ for some $p\in\partial_{QG}X$.

Choose $x_n\in X$ diverging to infinity such that $f(x_n) > 1/2$ for each $n\ge 1$.
Using \ref{SequenceVsRayLemma} detect a subsequence of
$\{x_n\}$ converging to the same point as an integral $(k,k)$-quasi-geodesic ray, a contradiction.
\end{proof}

\begin{Observation}
If one changes the definition of the set $S_n$ above as that of all integral $(k,k)$-quasi-geodesic defined on
$[1,m]\cap Z_+$, $m\leq n$, and starting at $x_0$,
then there are retractions $r^{n+1}_n:S_{n+1}\to S_n$ defined
by $r^{n+1}_n(g):=g$ if the domain of $g$ is $[1,m]\cap Z_+$, $m\leq n$, and by
$r^{n+1}_n(g):=g|[1,n]\cap Z_+$ if the domain of $g$ is $[1,n+1]\cap Z_+$, $m\leq n$. Its inverse limit
maps continuously onto $X_{QG}$.
\end{Observation}

\begin{Corollary}
Suppose $k\ge 1$ and $X$ is a metric space such that for some $M > 0$ every bounded subset of $X$ can be covered by finitely many sets of diameter less than $M$.
If $X$ is integrally $(k,k)$-quasi-geodesic, then the quasi-geodesic boundary of $X$ is compact metrizable and every  point on it is the limit of 
some  integral
$(k+2M,k+2M)$-quasi-geodesic ray $h$ based at $x_0$.
Moreover, if $X$ is proper, then the quasi-geodesic compactification of $X$ is compact metrizable.
\end{Corollary}
\begin{proof}
Pick a maximal subset $Y$ of $X$ containing $x_0$ such that $d(x,y)\ge M$
if $x\ne y\in Y$. Every integral $(k,k)$-quasi-geodesic $g$ in $X$
is approximable by $h$ with values in $Y$ such that $d(h(i),g(i)) < M$ for all $i$
in the domain of $g$. Notice $h$ is an integral $(k+2M,k+2M)$-quasi-geodesic in $Y$. The closure of $Y$ in the geometric quasi-geodesic compactification of $X$ is equal to the quasi-geodesic compactification of $Y$, hence is compact metrizable
and $\partial_{QG}X=\partial_{QG}Y$.

To see that the quasi-geodesic compactification of $X$ is compact metrizable, we choose a countable basis $\{U_n\}_{n\ge 1}$ of open subsets of
$\partial_{QG}X$. If $cl(U_m)\subset int(U_n)=\emptyset$,
choose open $V_{m,n}$ 
extending $U_m$ such that $\partial_{QG}X\cap cl(V_{m,n})\subset
U_n$ and it misses $B(x_0,m+n)$.
Enlarge a countable basis of $X$ consisting of bounded subsets
by adding all finite intersections of sets of the form $V_{m,n}$.
That family is a countable basis of open sets for $X_{QG}$, hence
$X_{QG}$ is metrizable. Indeed, if $p\in \partial_{QG}X$ is contained in an open set $W$ of $X_{QG}$, then the intersection
of all $cl(V_{m,n})$, $p\in V_{m,n}$ is first contained in $\partial_{QG}X$ and then it is contained in the intersection of all
$U_n$ that contain $p$. Hence, that intersection equals $p$,
so a finite intersection must be a subset of $W$ due to compactness of $X_{QG}$.
\end{proof}

\begin{Corollary}\label{PointsViaGeodesicRays}
Given a proper geodesic space $X$ every geodesic ray $g$ in $X$ converges to a point of the boundary $\partial_{QG}X$. If $x_0\in X$ and  $p\in\partial_{QG}X$, then there is a geodesic ray in $X$ starting at $x_0$ and converging to $p$ in $X_{QG}$.
\end{Corollary}
\begin{proof}
If $p\in\partial_{QG}X$, choose a sequence $x_n$, $n\ge 1$, in $X$
converging to $p$ in $X_{QG}$. Use
\ref{ExistenceOfGeodesicRays}.
\end{proof}

\begin{Lemma}\label{SuffCondForBoundaryPoints}
Two integral quasi-geodesics $g$ and $h$ converge to the same point in the quasi-geometric boundary if there is $k$ such that both are $(k,k)$-quasi-geodesics and every cone
$C(x_0,k,k,g,m)$ contains infinitely many points of $h$.
\end{Lemma}
\begin{proof}
Suppose $f:X\to R$ is a geometrically slowly-oscillating function
such that $f(g)\ne f(h)$. Choose $\epsilon > 2\cdot |f(g)-f(h)$.
Choose $m\ge 1$ such that the diameter of $f(C_X(x_0,k,k,g,m))$
is smaller than $\epsilon$ the closure of that set contains both $f(g)$
and $f(h)$, a contradiction.
\end{proof}

\begin{Corollary}
Two integral quasi-geodesics $g$ and $h$ to converge to the same point in the boundary if their Hausdorff distance is finite.
\end{Corollary}
\begin{proof}
Choose $k\ge 1$ such that both $g$ and $h$ are integral $(k,k)$-quasi-geodesics. Suppose the Hausdorff distance from $g$ to $h$ is less than $M$. Notice that $C_X(x_0,k+2M,k+2M,g,m)$ contains infinitely many points of $h$ and apply \ref{SuffCondForBoundaryPoints}.
\end{proof}

\begin{Corollary}\label{QGBdOfPlane}
The quasi-geodesic boundary of the plane is just one point.
The same is true for $\mathbb{Z}_+\times \mathbb{Z}_+$.
\end{Corollary}
\begin{proof}
Either use \ref{ConeExample} or notice there is an integral quasi-geodesic on the plane intersecting all the rays emanating from the origin infinitely many times and apply \ref{SuffCondForBoundaryPoints}.
\end{proof}

\begin{Example}
The quasi-geometric boundary of the universal cover of $R^2\vee S^1$ is two sequences converging to two different points.
Indeed, that space is a sequence of flats intersecting reals at integers,
so every integral quasi-geodesic either settles in one of the flats or spends bounded time in each of them eventually either moving right or left on the real line.
\end{Example}

\begin{Proposition}\label{QGBdAreHomeo}
If $X$ and $Y$ are quasi-isometric, then their quasi-geodesic boundaries are homeomorphic.
\end{Proposition}
\begin{proof}
By Lemma \ref{DiscreteSubsetLemma}, we may assume $X$ and $Y$ are $1$-discrete and there is a surjective coarse equivalence $f:X\to Y$. Moreover, by picking one point from each fiber of $f$, we may reduce the general case to that of $f$ being a bijection.

Now, $g:Y\to R$ is geometrically slowly-oscillating if and only if $g\circ f$ is geometrically slowly-oscillating. Also,
$a:(a,b)\cap Z_+\to X$ is a quasi-geodesic if and only if
$f\circ a$ is one.
\end{proof}

\begin{Corollary}\label{QIEmbeddingOnQGBoundaries}
Every continuous quasi-isometric embedding $f:X\to Y$ extends to a continuous 
map on geometric quasi-geodesic compactification that restricts to a map of quasi-geometric boundaries.
\end{Corollary}
\begin{proof}
Given a geometrically slowly-oscillating function $a:Y\to R$,
$a\circ f$ is also geometrically slowly-oscillating. Hence,
$f$ extends to $\tilde f:X_{QG}\to Y_{QG}$.
Now, use a similar argument as in  \ref{QGBdAreHomeo}.
\end{proof}

Given a metric space $(X,d)$, the \textbf{Gromov product} of $x$ and $y$ with respect to $a\in X$ is defined by
$$
\left< x,y\right>_a=\frac{1}{2}\big(d(x,a)+d(y,a)-d(x,y)\big).
$$

Recall that metric space $(X,d)$ is (Gromov) $ \delta-$\textbf{hyperbolic} if it satisfies the $\delta/4$-inequality:
$$
\left< x,y\right>_{a} \geq \min \{\left< x,z\right>_{a},\left< z,y \right> _{a}\}-\delta/4, \quad \forall x,y,z,a\in X.$$

$(X,d)$ is \textbf{Gromov hyperbolic} if it is $ \delta-$hyperbolic for some $\delta > 0$.

\begin{Proposition}\label{HypBdEqualsQGBd}
If $X$ is proper and Gromov hyperbolic, then the Gromov boundary of $X$ equals the quasi-geodesic boundary of $X$.
\end{Proposition}
\begin{proof}
Recall (see \cite{CS})  that a (quasi-)geodesic $\gamma$ is called \textbf{M--Morse} if for
any constants $K\ge 1$, $L\ge 0$, there is a constant $M = M(K, L)$, such that for every
$(K, L)$-quasi-geodesic $\sigma$ with endpoints on $\gamma$, we have that $\sigma$ is contained in the $M$-neighborhood of $\gamma$.

The fundamental property of Gromov hyperbolic spaces is expressed
in Theorem 1.3.2 of \cite{BS}:\\
  (Stability of geodesics). Let $X$ be a $\delta$-hyperbolic geodesic space, $a\ge 1, b\ge 0$. There exists $H = H(a,b,\delta) > 0$ such that for every $n\in N$ the image of every $(a,b)$-quasi-isometric map $f : {1,...,n}\to X$, $im(f)$, lies in the $H$-neighborhood of any geodesic $c : [0,l]\to X$ with $c(0) = f(1), c(l) = f(n)$, and vice versa, $c$ lies in the $H$-neighborhood of $im(f)$.

What we need to prove is that every continuous geometrically slowly-oscilating function $f:X\to [a,b]$ extends over the Gromov compactification of $X$ and conversely: given a continuous 
function $f$ on the Gromov compactification of $X$, its restriction to $X$ is geometrically slowly-oscillating.

Suppose $f:X\to [a,b]$ is continuous and geometrically slowly oscillating. Given a sequence $\{x_i\}$ in $X$ converging to a point on the Gromov boundary, we need to show $f(x_i)$ converges
to a point in $[a,b]$.
$\{x_i\}$ in $X$ converging to a point on the Gromov boundary means $\lim\limits_{i\to\infty,j\to\infty}\left< x_i,y_j\right>_{p}=\infty$. Moreover, we may assume there is a sequence
of geodesics $g_i$ from $p$ to $x_i$ pointwise converging to
a geodesic ray $g$ at $p$. Apply Lemma \ref{BoundedlyApproachesProp}.

Conversely: given a continuous 
function $f$ on the Gromov compactification of $X$, its restriction to $X$ is geometrically slowly-oscillating.
Indeed, if $g$ is a in integral $(k,k)$-quasi-geodesic ray in $X$ at $x_0$ and $x_m\in C_X(x_0,k,k,g,m)$ for $m\ge 1$,
then applying Stability of geodesics we can see $\lim\limits_{i\to\infty,j\to\infty}\left< x_i,y_j\right>_{p}=\infty$.
Thus, $f(x_m)$ converges to a point and $f$ is geometrically slowly-oscillating.
\end{proof}

\begin{Proposition}\label{MapOfCAT0Space}
If $X$ is a proper CAT(0)-space, then there is a continuous surjection from the compactification of $X$ obtained by adding the visual boundary to the quasi-geodesic compactification of $X$.
\end{Proposition}
\begin{proof}
Fix $x_0\in X$. Given a point at the visual boundary of $X$ represented by a geodesic ray originating at $x_0$, $g(t)$, $t\ge 0$, converges to a point $[g]$ of the quasi-geodesic boundary of $X$ by
\ref{PointsViaGeodesicRays} and sending $g$ to $[g]$ is a surjection of the boundaries.

Suppose a sequence $y_n$ of points in $X$ converges to a point in the visual boundary represented by a geodesic ray $g$ originating at $x_0$. What it means is that 
$g_n(t)$ converges to $g(t)$ for each $t\ge 0$, where $g_n$ is the geodesic from $x_0$ to $y_n$ for each $n\ge 1$.
Reduce $g$ to a $(1,1)$-quasi-geodesic by considering only the values $g(k)$, $k\in \mathbb{Z}_+$. Similarly, reduce each $g_n$
to a $(1,1)$-quasi-geodesic. Notice $g_n$ boundedly approaches
$g$.

Given any geometrically slowly-oscillating function $f:X\to \mathbb{C}$, \ref{BoundedlyApproachesProp} says
$f(y_n)\to f(g)$. Therefore, $y_n\to [g]$.
\end{proof}

\begin{Remark}
Notice that Croke-Kleiner \cite{CK} examples contain lots of flats (subspaces isometric to the plane) that produce circles in the visual boundary. Therefore, those circles are mapped to points in the quasi-geodesic boundary.
\end{Remark}

\section{Spaces with trivial quasi-geodesic boundary}

\begin{Proposition}\label{QGBdOfProduct}
If $X$ and $Y$ are proper unbounded integral quasi-geodesic spaces, then their cartesian product has trivial quasi-geodesic boundary, i.e. it consists of just one point.
\end{Proposition}
\begin{proof}
As in \ref{ProductOfQuasiGeodesicSpaces}, we consider $X\times Y$ with the $l_1$-metric.

Pick $x_0\in X$ and $y_0\in Y$. Given any integral $(k,k)$-quasi-geodesics $f$ in $X$ and $g$ in $Y$ originating respectively at $x_0$ and $y_0$, the function 
$h:\mathbb{Z}_+\times \mathbb{Z}_+\to X\times Y$ defined by $h(t,s)=(f(t),g(s))$ is a quasi-isometric embedding of a "quadrant" into $X\times Y$.
As the quasi-geodesic boundary of the integral quadrant is trivial by \ref{QGBdOfPlane},
all basic integral quasi-geodesics in $X\times Y$ converge to the same point $p$ of $\partial_{QG}(X\times Y)$ by \ref{QIEmbeddingOnQGBoundaries}, where by basic integral quasi-geodesic in $X\times Y$ we mean one of the form $t\to (f(t),y_0)$ or $t\to (x_0,g(t))$.

Suppose a sequence of points $(x_n,y_n)$ converges
to $q\ne p$ in  $\partial_{QG}(X\times Y)$. Consider integral $(k,k)$-quasi-geodesics $f_n$ from $(x_0,y_0)$ to $(x_n,y_n)$
as in the proof of \ref{ProductOfQuasiGeodesicSpaces}.
We may assume as in the proof of \ref{SequenceVsRayLemma}
that $f_n$ are boundedly approaching an integral quasi-geodesic ray $f$ in $X\times Y$. If the sequence $\{x_n\}$ is not diverging to infinity, then $f$ is of finite Hausdorff distance from a basic integral quasi-geodesic ray, hence it converges to $q$, a contradiction.
However, if $\{x_n\}$ diverges to infinity, then $f$ is of finite Hausdorff distance from a basic integral quasi-geodesic ray, a contradiction as well.
\end{proof}

The purpose of the next result is to show that the well-known examples of Croke-Kleiner \cite{CK} have the trivial quasi-geodesic boundary. We use a description from \cite{CM}  so that our proof also applies to generalized Croke-Kleiner spaces constructed by Mooney.

\begin{Proposition}\label{CKSpacesQGBoundaries}
Croke-Kleiner spaces have trivial quasi-geodesic boundaries.
\end{Proposition}
\begin{proof}
By \textbf{Croke-Kleiner spaces} we mean $CAT(0)$-spaces constructed in \cite{CK} 
that are quasi-isometric to the group $\{a,b,c,d | [a,b], [b,c], [c,d]\}$. One of them, call it $X$, can be represented as the universal cover of a union of three tori $T_1, T_2, T_3$ and has the following structure
(see \cite{CM} where \textbf{generalized Croke-Kleiner spaces} are described in a similar way):\\
a. $X$ is the union of collection of closed convex subspaces, called \textbf{blocks} that are isometric to the product of a tree and a line. Hence, the visual boundary $\partial B$ of every block
$B$ is the suspension of a Cantor set and the quasi-geodesic boundary of $B$ is trivial by \ref{QGBdOfProduct}.
The suspension points are called \textbf{poles}. The intersection of two blocks is a Euclidean plane called a \textbf{wall}. \\
b. The nerve $N$ of the collection of blocks is a tree.\\
c. Let $B_0$ and $B_1$ be blocks, and $D$ be the distance between the corresponding vertices in N. Then:\\
(1) If $D = 1$, then $\partial B_0\cap \partial B_1 = \partial W$ where $W$ is the wall $B_0\cap B_1$.\\
(2) If $D=2$, then $\partial B_0\cap \partial B_1$ is the set of poles of $B_{1/2}$, where $B_{1/2}$ intersects $B_0$ and $B_1$.\\
(3) If $D > 2$, then $\partial B_0\cap \partial B_1=\emptyset$.\\
d. The union $Y$ of block boundaries in $\partial X$ is dense in $\partial X$.

Consider the natural map $p:\partial X\to \partial_{QG}X$.
Each $\partial B$ is sent to a point. By (1) above all those points are equal, hence $Y$ is sent to a point. As $Y$ is dense in $\partial X$,
the image of $\partial X$ is a point.
\end{proof}

\section{Sublinear quasi-geodesic boundaries}

In coarse theory (see \cite{Grom}, \cite{Roe lectures}, \cite{DK} or \cite{BH}) the most prominent boundary of a metric space is the Higson corona.
However, there is another coarse boundary, namely the sublinear Higson corona (see \cite{DS}  or \cite{CDSV}) which is usually smaller than Higson corona. In this section we introduce boundaries of geodesic spaces related to the sublinear Higson corona. Potentially, they may be related to the sublinear Morse boundaries of Tiozzo-Qing-Rafi
\cite{TQR1}  and \cite{TQR2}.

\begin{Definition}
As in \cite{TQR2} (see Section 2.1) a \textbf{sublinear function}
is a function $\kappa:[0,\infty)\to [1,\infty)$ such that
$$\lim\limits_{t\to\infty}\frac{\kappa(t)}{t}=0.$$
\end{Definition}
Notice that we should not quibble about the domain of $\kappa$.
It suffices it contains $[a,\infty)$ for some $a\in R$.

Sublinear functions are essentially equivalent to \textbf{asymptotically sublinear functions} in the terminology of \cite{DS} or \cite{CDSV},
i.e. functions $s:R_+\to R_+$ such that for each non-constant linear function $f:R_+\to R_+$ there is $r > 0$ so that
$s(x)\leq f(x)$ for all $x > r$.

The sublinear Higson corona of a metric space is defined
in \cite{DS} via a coarse structure  ( see also \cite{CDSV}).
We are going to introduce it in analogy to slow-oscillating functions:
\begin{Definition}\label{SublinearlySlowOsc}
Given a metric space $X$ and given $x_0\in X$ define
$||x||$ as $d(x,x_0)$. A function $f:X\to \mathbb{C}$
is \textbf{sublinearly slowly-oscillating} if, whenever $\kappa$ is a sublinear function and $\epsilon > 0$,
there is a bounded subset $B$ of $X$ such that
$x,y\in X\setminus B$
and $d(x,y)\leq \kappa(||x||)$,
then $|f(x)-f(y)|\le \epsilon$.

The \textbf{sublinear Higson corona} of $X$ is defined analogously to the Higson corona but, instead of slowly-oscillating function, one uses sublinearly slowly-oscillating functions.
\end{Definition}

Since we want to create a concept analogous to $\kappa$-boundaries of \cite{TQR2}, we will create a larger classes of functions than in
\ref{SublinearlySlowOsc}:
\begin{Definition}
Given a metric space $X$, $x_0\in X$, and a sublinear function $\kappa$, declare a function $f:X\to \mathbb{C}$
to be \textbf{$\kappa$-slowly-oscillating} if, whenever $C, \epsilon > 0$,
there is a bounded subset $B$ of $X$ such that
$x,y\in X\setminus B$
and $d(x,y)\leq C\kappa(||x||)$,
then $|f(x)-f(y)|\le \epsilon$. 
\end{Definition}

\begin{Definition}
Given a proper metric space $X$ and a sublinear function $\kappa$
define the \textbf{$\kappa$-quasi-geodesic compactification}
$X_{\kappa QG}$ of $X$ to be the compactification induced
by all bounded continuous functions $f:X\to \mathbb{C}$ that are
geometrically slowly-oscillating and $\kappa$-slowly-oscillating.
\end{Definition}

Notice there is a continuous extension $X_{QG}\to X_{\kappa QG}$
of the identity map $id_X:X\to X$,
so $X_{\kappa QG}$ is compact metrizable if $X$ is integrally quasi-geodesic.

\begin{Definition}
Given a proper metric space $X$
define the \textbf{sublinear quasi-geodesic compactification}
$X_{sQG}$ of $X$ to be the compactification induced
by all bounded continuous functions $f:X\to \mathbb{C}$ that are
geometrically slowly-oscillating and  sublinearly slowly-oscillating.
\end{Definition}

Again, notice there is a continuous extension $X_{QG}\to X_{sQG}$
of the identity map $id_X:X\to X$,
so $X_{sQG}$ is compact metrizable if $X$ is integrally quasi-geodesic.

Tiozzo-Qing-Rafi \cite{TQR2} (see Definition 3.3) define
two quasi-geodesic rays $g$ and $h$ based at $x_0$ to 
\textbf{sublinearly
track each other} if
$$\lim\limits_{t\to\infty}\frac{d_X(g(t),h(t))}{t}=0.$$
In other words, $\kappa(t):=d_X(g(t),h(t))$ is a sublinear function.

That leads to an equivalence relation $g\sim h$ on the space of
quasi-geodesics in $X$ based at $x_0$ and those equivalence classes
represent points in the boundaries constructed in \cite{TQR2}.
We can only say that each equivalence class maps to a single point in the sublinear quasi-geometric boundary of $X$.

In the case of $\kappa$-quasi-geodesic boundary we can define
$\kappa$-tracking in two ways and they are equivalent:

\begin{Lemma}\label{ConcaveEquivalenceOfTwoWaysToTrack}
Suppose $X$ is a metric space and $g,h$ are two integral quasi-geodesics in $X$. Given a sublinear concave and increasing function $\kappa$, the following two conditions are equivalent:\\
1. There is a constant $C \ge 1$
such that $d_X(g(t),h(t))\leq C\cdot \kappa(||g(t)||)$ for sufficiently large $t$. \\
2. There is a constant $M > 0$
such that $d_X(g(t),h(t))\leq M\cdot \kappa(t)$ for sufficiently large $t$. 
\end{Lemma}
\begin{proof}
Notice (see Lemma 2.4 in \cite{TQR2}), $\kappa$
satisfies $\kappa(\lambda\cdot t)\leq \lambda\cdot \kappa(t)$ for all $\lambda \ge 1$. 

1)$\implies$2). Since both $||g(t)||$ and $||h(t)||$ are bounded
by some linear function $m\cdot t+b\leq (m+b)\cdot t$ for $t\ge 1$, where $m,b > 1/C$, $\kappa(C\cdot ||g(t)||)\leq (m+b)\cdot C\cdot \kappa(t)$ for $t\ge 1$.
 
2)$\implies$1). Since $||g(t)||$ is at least
 $(2/m)\cdot t-b\ge (1/m)\cdot t$ for $t\ge m\cdot b$, where $m,b > 1$, $m\cdot \kappa(||g(t)||)\ge \kappa(m\cdot ||g(t)||)\ge \kappa(t)$ for $t\ge b\cdot m$.
\end{proof}

\begin{Lemma}\label{EquivalenceOfTwoWaysToTrack}
Suppose $X$ is a metric space and $g,h$ are two integral quasi-geodesics in $X$. Given a sublinear function $\kappa$, the following two conditions are equivalent:\\
1. There is a constant $C \ge 1$
such that $d_X(g(t),h(t))\leq C\cdot \kappa(C\cdot ||g(t)||)$ for sufficiently large $t$. \\
2. There is a constant $C \ge 1$ and a constant $K \ge 1$
such that $d_X(g(t),h(t))\leq C\cdot \kappa(K\cdot ||g(t)||)$ for sufficiently large $t$. \\
3. There is a constant $M > 0$
such that $d_X(g(t),h(t))\leq M\cdot \kappa(t)$ for sufficiently large $t$. 
\end{Lemma}
\begin{proof}
See Remark 2.3 in \cite{TQR2}),
where it is shown that there exists a constant $Q\ge 1$ and an increasing and concave sublinear function $\bar\kappa$ such that
$\kappa(t)\le \bar\kappa(t)\le Q\cdot \kappa(t)$.

1)$\implies$2) is obvious.\\
2)$\implies$3). By \ref{ConcaveEquivalenceOfTwoWaysToTrack},
there is a constant $M > 0$
such that $d_X(g(t),h(t))\leq M\cdot \bar\kappa(t)\leq M\cdot Q\cdot \kappa(t)$ for sufficiently large $t$. 
 
3)$\implies$1). By \ref{ConcaveEquivalenceOfTwoWaysToTrack},
there is a constant $C \ge 1$
such that $d_X(g(t),h(t))\leq C\cdot \bar\kappa(||g(t)||)
\leq Q\cdot C\cdot \kappa(||g(t)||)$ for sufficiently large $t$. 
\end{proof}

Now, we can define two quasi-geodesic rays $g$ and $h$ based at $x_0$ to 
\textbf{$\kappa$-track each other} if there is a constant $C > 1$
such that $d_X(g(t),h(t))\leq C\cdot \kappa(||g(t)||)$ for sufficiently large $t$. The resulting equivalence classes (use \ref{EquivalenceOfTwoWaysToTrack}) map to points of the $\kappa$-quasi-geodesic boundary of $X$.

In the case of Gromov hyperbolic spaces all sublinear quasi-geodesic boundaries are equal to the Gromov boundary:

\begin{Proposition}\label{HypBdEqualssQGBd}
If $X$ is proper and Gromov hyperbolic, then the Gromov boundary of $X$ equals the sublinear quasi-geodesic boundary of $X$.
\end{Proposition}
\begin{proof}
Let $\kappa$  be a sublinear function and suppose two geodesic rays $g$ and $h$ based at $x_0$ 
$\kappa$-track each other. We need to show that $g$ and $h$ are within bounded distance. Given $M > 0$ we plan to show that $d(g(t),h(t))\leq \delta$ for all $t\leq M$. For that
it suffices to find $q > 0$ such that $\langle g(q),h(q)\rangle_{x_0} > M$. Pick a constant $C \ge 1$
such that $d_X(g(t),h(t))\leq C\cdot \kappa(C\cdot ||g(t)||)$ for sufficiently large $t$ (see \ref{EquivalenceOfTwoWaysToTrack}).
Now, we can find $q$ such that $||g(q)||=||h(q)||\ge 4M$ and
$C\cdot \kappa(C\cdot ||g(q)||)\leq  ||g(q)|| $.
Consequently, $2\cdot\langle g(q),h(q)\rangle_{x_0}
=||g(q)+||h(q)||-d(g(q),h(q)) > 4M  > 2\cdot M$.

Similarly, one can show that if $x_n\to\infty$ and $d(x_n,y_n)\leq C\cdot \kappa(||x_n||)$ for each $n\ge 1$, then 
$\langle x_n,y_n\rangle_{x_0}\to\infty$, so any continuous complex-valued
function $f$ on the Gromov compactification of $X$ restricts to
a $\kappa$-slowly-oscillating function on $X$.
\end{proof}

\section{Quasi-geodesic ends of spaces}\label{EndsOfQGeoSpaces}

In \cite{DR} the authors constructed a theory of ends of spaces via linear algebra. It makes sense to compare ends of $X$ using its quasi-geodesic boundary to the coarse ends of $X$.

\begin{Proposition}\label{CoarseEndsVsQGBoundary}
a.  If $X$ is proper metric, then the space of coarse ends of $X$ embeds into the space of components of the quasi-geodesic boundary of $X$.\\
b. If $X$ is proper metric and quasi-geodesic, then the space of components of the quasi-geodesic boundary of $X$ is identical with the space of coarse ends of $X$.
\end{Proposition}
\begin{proof}
a. The space of coarse ends of $X$ is identical with the space of components of its Higson corona (see \cite{DR}). Those components can be identified using slowly oscillating functions $f:X\to [0,1]$ such that when extended over the Higson compactification of $X$ there are only two values on a neighborhood of the Higson corona. That means $f$ has properties as in \ref{CoarselyClopenGeoSlo}, so it is geometrically slowly oscillating and it extends over the quasi-geodesic boundary so that it induces a map of its components.

b. By \ref{GeoSloOscAreSlowOscInGeodesicSpaces}
any geometrically slowly oscillating function $f:X\to [0,1]$
is slowly osillating, so it extends over the quasi-geodesic compactification of $X$. That means there is a continuous extension
of $id_X$ from the Higson compactification $h(X)$ of $X$ onto
the quasi-geodesic compactification of $X$. In particular, the space of coarse ends of $X$ maps onto the space of ends of the quasi-geodesic boundary of $X$.
\end{proof}

\begin{Corollary}\label{QGBdOfCAT0Spaces}
If $X$ is a proper CAT(0)-space with totally disconnected visual boundary, then the quasi-geodesic boundary of $X$ is identical with the visual boundary of $X$.
\end{Corollary}
\begin{proof}
In that case the visual boundary of $X$ can be identified with coarse ends of $X$ (see \cite{DR}). Since there is a natural surjection from the visual boundary of $X$ onto $\partial_{QG} X$ (see \ref{MapOfCAT0Space}), by \ref{CoarseEndsVsQGBoundary} that map is a homeomorphism.
\end{proof}

\begin{Example}
The sublinear quasi-geodesic boundary of $X$ may have less ends than the quasi-geodesic boundary of $X$:\\
Let $a(x):=\sqrt{x}$. Consider the set of all points on the plane
of the form $(n,0)$ or $(0,n)$, where $n\ge 0$ is an integer.

Define the distance from $(x,0)$ to $(0,x)$ to be $a(x)$ and 
extend it to the distance between $(x,0)$ to $(0,y)$, $x< y$,
to be $a(y)+y-x$. If $x > y$, we define it to be $a(x)+x-y$.

Both quasi-geodesic rays are different in the quasi-geodesic boundary but are the same in the sublinear quasi-geodesic boundary.

However, if we add bridges from $(n^2,0)$ to $(0,n^2)$ of length $n$ and exyend the metric naturally, then both
quasi-geodesic rays are identical in the quasi-geodesic boundary.
\end{Example}

We do not know of a proper geodesic space $X$ such that its space sublinear quasi-geodesic ends, i.e. the space of components of the sublinear quasi-geodesic boundary of $X$, is smaller than the space of coarse ends of $X$.


\begin{thebibliography}{99}
\bibitem{BH} M. Bridson and A. Haefliger, \emph{Metric spaces of non-positive curvature}, Springer- Verlag, Berlin, 1999.

\bibitem{BS} S. Buyalo and V. Schroeder, \emph{Elements of asymptotic geometry}, EMS Monographs in Mathematics, European Mathematical Society, Zurich (2007).

\bibitem{CM} Christopher H. Cashen and John M Mackay, \emph{A metrizable topology on the contracting boundary of a group}, Transactions of the American Mathematical Society, 372(3):1555--1600.

\bibitem{CS} Ruth Charney and Harold Sultan, \emph{Contracting boundaries of CAT(0) spaces}, Journal of Topology, 8(1):93--117, 09 2014.

  \bibitem{CDSV}   M.Cencelj, J.Dydak, J.Smrekar,  and A.Vavpeti\v c, {\em    Sublinear Higson corona and Lipschitz extensions}, Houston Journal of Mathematics 37 (2011), pp. 1307--1322


\bibitem{CK} Christopher Croke and Bruce Kleiner, \emph{Spaces with nonpositive curvature and their ideal boundaries}, Topology, 39:549--556, 1998.

\bibitem{CM} Christopher Mooney, \emph{Generalizing the Croke-Kleiner Construction}, 	arXiv:0807.4745 [math.GR]

\bibitem{DS} A.N.Dranishnikov and J.Smith, On asymptotic Assouad-Nagata dimension, preprint ArXiv:
math.MG/0607143

\bibitem{DK} C. Drutu, M. Kapovich, \emph{Geometric group theory}, Colloquium publications, Vol. 63, American Mathematics Society (2018).

\bibitem{Dyd}  J.Dydak, {\em    Covariant and contravariant approaches in topology},    International Journal of Mathematics
and Mathematical Sciences
   20 (1997),  621--626.

\bibitem{DR} Jerzy Dydak and Hussain Rashed, \emph{Ends of spaces via linear algebra}, 	arXiv:2206.08151 [math.MG]

\bibitem{DW}
Jerzy Dydak, Thomas Weighill, \emph{Extension Theorems for Large Scale Spaces via Coarse Neighbourhoods}, Mediterranean Journal of Mathematics, (2018) 15: 59.

\bibitem{Grom}
M. Gromov, \emph{Asymptotic invariants for infinite groups}, in
Geometric Group Theory, vol. 2, 1--295, G. Niblo and M. Roller,
eds., Cambridge University Press, 1993.

\bibitem{MC} Matthew Cordes, \emph{Morse boundaries of proper geodesic metric spaces}, Groups Geom. Dyn. 11 (2017), no. 4, pp. 1281--1306


\bibitem{MM} Marston Morse,\emph{ A fundamental class of geodesics on any closed surface of genus greater than one},
Trans. Amer. Math. Soc. 26 (1924), 25--60.


\bibitem{Roe lectures}
J. Roe, \emph{Lectures on coarse geometry}, University Lecture
Series, 31. American Mathematical Society, Providence, RI, 2003.

\bibitem{TQR1} Giulio Tiozzo, Yulan Qing, Kasra Rafi, \emph{Sublinearly Morse Boundary I: CAT(0) Spaces}, 	arXiv:1909.02096 [math.GT]

\bibitem{TQR2} Giulio Tiozzo, Yulan Qing, Kasra Rafi, \emph{Sublinearly Morse Boundary II: Proper Geodesic Spaces}, 	arXiv:2011.03481 [math.GT]



\end{thebibliography}
\end{document}